\documentclass[12pt,a4paper]{article}

\usepackage{latexsym,amssymb}

\begin{document}

\title{A Note on Near-factor-critical Graphs}
\author{
Kuo-Ching Huang \thanks{Research supported by NSC 
(No. 102-2115-M-126-002).}\\
\normalsize Department of Financial and Computational Mathematics\\
\normalsize Providence University\\
\normalsize Shalu 43301, Taichung, Taiwan\\
\normalsize {\tt Email:kchuang@gm.pu.edu.tw}
\and
Ko-Wei Lih \thanks{Research supported by NSC 
(No. 102-2115-M-001-010).}\\
\normalsize Institute of Mathematics\\
\normalsize Academia Sinica\\
\normalsize Taipei 10617, Taiwan\\
\normalsize {\tt Email:makwlih@sinica.edu.tw}
}

\date{\small \today}

\maketitle


\newtheorem{define}{Definition}
\newtheorem{theorem}[define]{Theorem}
\newtheorem{lemma}[define]{Lemma}
\newtheorem{corollary}[define]{Corollary}
%
%
\newenvironment{proof}{
\par
\noindent {\bf Proof.}\rm}%
{\mbox{}\hfill\rule{0.5em}{0.809em}\par}
\newenvironment{remark}{
\par
\noindent {\bf Remark.}\rm}

\baselineskip=16pt
\parindent=0.8cm

\begin{abstract}

\noindent
A near-factor of a finite simple graph $G$ is a matching that 
saturates all vertices except one. A graph $G$ is said to be  
near-factor-critical if the deletion of any vertex from $G$ 
results in a subgraph that has a near-factor. We prove that a 
connected graph $G$ is near-factor-critical if and only if it 
has a perfect matching. We also characterize disconnected 
near-factor-critical graphs.

\bigskip


\medskip

\noindent
Mathematical Subject Classification (MSC) 2010: 05C70
\end{abstract}

%
\section{Introduction}
%

All graphs $G$ considered in this note are finite and simple with 
vertex set $V(G)$ and edge set $E(G)$. The number of vertices of 
$G$ is called the {\em order} of $G$. A {\em matching} $M$ of $G$ 
is a subset of $E(G)$ such that no two edges in $M$ share a common
endpoint. A matching $M$ {\em saturates} a vertex $v$ of $G$, or 
$v$ is said to be $M$-{\em saturated}, if $v$ is an endpoint of 
some edge in $M$. Otherwise, $v$ is said to be $M$-{\em unsaturated}. 
A matching $M$ is called {\em perfect} if every vertex of $G$ is 
$M$-{\em saturated}. A $1$-{\em factor} is synonymous with a perfect 
matching. A {\em near-factor} is a matching that saturates all vertices 
except one.

Let $S \subseteq V(G)$. The subgraph of $G$ obtained from $G$ by deleting 
all vertices of $S$ is denoted by $G \setminus S$. In particular, if $S$ 
is a singleton $\{v\}$, we denote $G \setminus \{v\}$ by $G-v$. A graph 
$G$ is said to be {\em factor-critical} if $G-v$ has a 1-factor for every 
$v \in V(G)$. A factor-critical graph is necessarily of odd order. This 
notion was first introduced by Gallai \cite{gallai} and has been intensively 
studied, e.g. \cite{lourao, match}. We call a graph $G$ {\em near-factor-critical} 
if $G-v$ has a near-factor for every $v \in V(G)$. A near-factor-critical 
graph $G$ is necessarily of even order.

Li et al. \cite{trinks} showed that, for a graph $G$ with a vertex of degree 
one, $G$ has a 1-factor if and only if $G$ is near-factor-critical. They 
asked whether this result could be generalized to any connected graph $G$.
In this note, we are going to give a positive solution to this question.
Since the union of two disjoint odd cycles has no 1-factor and is 
near-factor-critical, the connectedness of $G$ is essential for a generalization.

%
\section{Main results}
%

A {\em component} of a graph $G$ is a maximal connected subgraph of $G$. 
An {\em odd}, or {\em even}, component is a component having odd, or even, 
number of vertices. The number of odd components is denoted by $o(G)$. 
If a graph $G$ has a near-factor $M$, then there exists a unique 
$M$-unsaturated vertex in $G$ and is denoted by $u(M)$.

\begin{lemma} \label{lem-1}
Let $G$ be a connected graph. Suppose that, for $S \subseteq V(G)$, 
there is an odd component $H$ of $G \setminus S$. For any vertex 
$v \notin V(H)$, if $M$ is a near-factor of $G-v$ and $u(M) \notin 
V(H)$, then there is an edge of $M$ joining a vertex of $S$ with 
a vertex of $H$.
\end{lemma}

\begin{proof}
Since $u(M) \notin V(H)$, every vertex of $H$ is $M$-saturated. 
Since $v \notin V(H)$, if each edge of $M$ has zero or two endpoints 
in $H$, then $H$ is of even order, a contradiction. By the connectedness 
of $G$, there must exist an edge of $M$ having one endpoint in $H$ and 
one endpoint in $S$.
\end{proof}

\begin{theorem}\label{main-1}
The following conditions are equivalent for a connected graph $G$.
\begin{enumerate}
\item
$G$ has a 1-factor.
\item
$o(G \setminus S) \leqslant |S|$ for all $S \subseteq V(G)$.
\item
$G$ is near-factor-critical.
\end{enumerate}
\end{theorem}

\begin{proof}
The equivalence between conditions 1 and 2 is the well-known Tutte's
Theorem \cite{tutte}. It is clear that condition 1 implies condition 3.
It remains to prove that condition 3 implies condition 2.

Suppose that condition 2 fails for $G$. Then there is a subset $S \subseteq 
V(G)$ such that $o(G \setminus S)> |S|$. Choose an arbitrary vertex $v \in S$. 
Then $G-v$ has a near-factor $M$. By Lemma \ref{lem-1}, $G \setminus S$ has 
at least $o(G \setminus S)-1$ odd components $H$ such that there is an edge 
of $M$ having one endpoint in $H$ and one endpoint in $S$. Thus, $|S| \geqslant 
1+o(G \setminus S)-1=o(G \setminus S)> |S|$, a contradiction.
\end{proof}

\begin{theorem}\label{main-2}
Let $G$ be a disconnected graph. Then $G$ is near-factor-critical if and 
only if one of the following holds.
\begin{enumerate}
\item
All components of $G$ are even and each of them has a 1-factor.
\item
There are only two components $H_1$ and $H_2$ of $G$ and each of 
them is factor-critical.
\end{enumerate}
\end{theorem}

\begin{proof}
The sufficiency is straightforward. We now prove the necessity.
Suppose that there exist an even component $F$ and an odd component 
$H$ of $G$. Choose an arbitrary vertex $v$ from $F$. Then $G-v$ has 
a near-factor $M$. Since both $F-v$ and $H$ have odd number of vertices,
each of them should contain an $M$-unsaturated vertex. This contradicts 
the uniqueness of $u(M)$. Therefore, $G$ consists of either all even or 
all odd components. In the latter case, the number of odd components is 
even since the order of $G$ is even.

Suppose that $G$ consists of even components $F_1, F_2, \ldots , F_p$, 
where $p \geqslant 2$. Choosing an arbitrary vertex $x$ of $F_1$, $G-x$ 
has a near-factor $M_1$. Since $F_1-x$ is of odd order, $u(M_1)$ belongs 
to $F_1-x$. Thus, $M_1$ restricted to $F_i$ is a 1-factor of $F_i$ for 
each $i \geqslant 2$. Similarly, $F_1$ can be argued to have a 1-factor 
if the vertex $x$ was chosen from $F_2$.

Next, suppose that $G$ consists of odd components $H_1, H_2, \ldots , 
H_{2q}$ for some $q \geqslant 1$. Choosing an arbitrary vertex $z$ of 
$H_1$, $G-z$ has a near-factor $M_2$ such that $u(M_2) \notin V(H_1)$ 
since $H_1-z$ is of even order. Each odd component $H_j$, $j \geqslant 
2$, should contain an $M_2$-unsaturated vertex. It follows that $q=1$ and 
$H_1-z$ has a 1-factor. Similarly, $H_2-z$ can be argued to have a 
1-factor if the vertex $z$ was chosen from $H_2$.
\end{proof}

\bigskip

\begin{remark}
It follows from Theorems \ref{main-1} and \ref{main-2} that the time 
complexity for recognizing near-factor-critical graphs is dominated 
by the complexity for recognizing factor-critical graphs. The latter 
was determined in \cite{lourao}, using the well-known maximum matching 
algorithm of \cite{max},  to run in $O(n^{1/2}m)$ time, where $n=|V(G)|$ 
and $m=|E(G)|$.
\end{remark}

%

\end{document}